\documentclass[12pt]{article}

\usepackage{enumerate}
\usepackage{amsthm,amsmath,amssymb}
\usepackage{graphicx}
\usepackage[colorlinks=true,citecolor=black,linkcolor=black,urlcolor=blue]{hyperref}
\usepackage[utf8]{inputenc}

\theoremstyle{plain}
\newtheorem{theorem}{Theorem}
\newtheorem{lemma}[theorem]{Lemma}
\newtheorem{corollary}[theorem]{Corollary}

\newtheorem{thm}{Theorem}
	
\theoremstyle{definition}
\newtheorem{definition}[theorem]{Definition}

\newtheorem{conjecture}[theorem]{Conjecture}

\newtheorem*{remark*}{Remark}

\theoremstyle{remark}

\newenvironment{thmrestate}[1]
  {\renewcommand{\thethm}{\ref{#1}$'$}%
   \addtocounter{thm}{-1}%
   \begin{thm}}
  {\end{thm}}

\title{\bf The (p,q)-extremal problem and the fractional chromatic number of\\ Kneser hypergraphs}

\author{Gabriela Araujo-Pardo\\
\small\tt garaujo@math.unam.mx
\and
 Juan Carlos Díaz-Patiño\\
\small\tt juancdp@im.unam.mx\\
\and
Luis Montejano\\
\small\tt luis@matem.unam.mx 
\and
Deborah Oliveros\\
\small\tt dolivero@matem.unam.mx\\
\and
\small Instituto de Matemáticas\\[-0.8ex]
\small Universidad Nacional Autónoma de México\\[-0.8ex] 
\small Ciudad Universitaria, 04510, México D.F.\\
}

\begin{document}

\maketitle

\begin{abstract}
The problem of computing the chromatic number of Kneser hypergraphs has been extensively studied over the last 40 years and the fractional version of the chromatic number of Kneser hypergraphs is only solved for particular cases. The \emph{$(p,q)$-extremal problem} consists in finding the maximum number of edges on a $k$-uniform hypergraph $\mathcal{H}$ with $n$ vertices such that among any $p$ edges some $q$ of them have no empty intersection. In this paper we have found a link between the fractional chromatic number of Kneser hypergraphs and the $(p,q)$-extremal problem and also solve the $(p,q)$-extremal problem for graphs if $n$ is sufficiently large and $p \geq q \geq 3$ by proposing it as a problem of extremal graph theory. With the aid of this result we calculate the fractional chromatic number of Kneser hypergraphs when they are composed with sets of cardinality 2.

\end{abstract}

\section{Introduction}
Let $p,q$ be natural numbers such that $p\geq q \geq 2$. A $k$-uniform hypergraph $\mathcal{H}$ has the $(p,q)$-\emph{property} if among any $p$ edges some $q$ of them have a common vertex. The \emph{$(p,q)$-extremal problem} consists in finding the maximum number of edges on a $k$-uniform hypergraph $\mathcal{H}$ with $n$ vertices that satisfies the $(p,q)$-property. The $(p,q)$-property has been extensively studied in discrete geometry. It was proposed by Hadwiger and Debrunner in 1956 (see \cite{hadwiger1957variante}), and studied by N.\ Alon and D.\ J.\ Kleitman \cite{alon1992piercing}. For an excellent survey of the topic see \cite{eckhoff2003survey}. The direction of our approach is related to several classical theorems of extremal graph theory and extremal set theory.

The $(p,2)$-extremal problem for $k$-uniform hypergraphs is equivalent to finding the maximum number of edges on a $k$-uniform hypergraph with no $p$ disjoint edges. This problem has been studied in several papers, and the cases of $k=2,3$ are completely solved (see \cite{erdHos1959maximal}, \cite{frankl2012maximum}, \cite{frankl2012maximumhypergraph} and \cite{luczak2012erdos}), and for the general case, there are partial results in \cite{erdHos1965problem}, \cite{frankl2013improved}, and \cite{huang2012size}.

The $(2,2)$-extremal problem is equivalent to the Erd\H{o}s-Ko-Rado Theorem (see \cite{erdos1961intersection}) and for $q\geq 2$ there exists a generalization of the Erd\H{o}s-Ko-Rado theorem given by P.\ Frankl \cite{frankl1976sperner} in a lemma that is equivalent to the $(q,q)$-extremal problem.

In this paper we solve the $(p,q)$-extremal problem for graphs and also give a lower bound for the $(p,q)$-extremal problem for $k$-uniform hypergraphs. The main result is the following:

\begin{theorem}[The $(p,q)$-extremal problem for graphs]\label{main}
Let $G$ be a graph with $n$ vertices satisfying the $(p,q)$-property; for $n\geq 2p^2$ and $p\geq q\geq 3$ we have
\[
|E(G)|\leq \binom{n}{2} -\binom{n-t}{2}+r,
\]
where $\displaystyle t=\left\lfloor \frac{p-1}{q-1}\right\rfloor$ and $r$ is the residue of $\displaystyle\frac{p-1}{q-1}$. This bound is sharp.
\end{theorem}

In 1955 Kneser posed a combinatorial problem (see \cite{kneser1955aufgabe}) that was later solved by Lovasz in 1978 using the Borsuk-Ulam Theorem, showing one of the earliest and most beautiful applications of topological methods in combinatorics (see \cite{lovasz1978kneser}). Lovasz's paper shows the equivalence between the Kneser conjecture and the chromatic number of Kneser graphs. There are several generalizations of the Kneser-Lovasz Theorem, one of them given by K.\ S.\ Sarkaria in 1990. He found the chromatic number of Kneser hypergraphs, also using topological tools (see \cite{sarkaria1990generalized}). Thus it is natural to ask for the fractional version of this result; i.e., what is the fractional chromatic number of the Kneser hypergraphs? The fractional chromatic number of a Kneser hypergraph can be computed in terms of the maximum number of edges of a hypergraph with the $(p,q)$-property, as we will see towards the end of this paper.

The paper is organized as follows: Section 2  presents some definitions and preliminary results; in Section 3, we do an extensive analysis of the $(p,q)$-extremal problem for graphs and prove the main theorem as well as analyzing some particular cases. Finally in Section 4 we define the fractional chromatic number of Kneser hypergraphs, and give the relationship between the $(p,q)$-extremal problem and the fractional chromatic number of the Kneser hypergraphs.

\section{Notation and Preliminaries}
A hypergraph $\mathcal{H}=(V(\mathcal{H}),E(\mathcal{H}))$ consists of a set $V(\mathcal{H})$ of vertices and a set $E(\mathcal{H})$ of edges, where each edge is a subset of $V(\mathcal{H})$. We denote the number of vertices and edges by $v(\mathcal{H})$ and $e(\mathcal{H})$ respectively. If $X$ and $Y$ are sets of vertices of $V(\mathcal{H})$ we denote by $E[X,Y]$ the set of edges of $\mathcal{H}$ with one vertex in $X$ and other in $Y$, and by $e(X,Y)$ their cardinality. If each edge of $\mathcal{H}$ is a $k$-subset of $V(\mathcal{H})$, we say that $\mathcal{H}$ is a \emph{$k$-uniform hypergraph}. If $v(\mathcal{H})=n$, it is convenient to just let $V(\mathcal{H})=[n]=\{1,2,\ldots,n\}$.  The family of all $k$-sets of $[n]$ is denoted by $\binom{[n]}{k}$, and for convenience, we write $\mathcal{H}\subseteq \binom{[n]}{k}$ to indicate that $\mathcal{H}$ is a $k$-uniform hypergraph on vertex set $[n]$. The hypergraph $\mathcal{F}$ is called a \emph{subhypergraph} of $\mathcal{H}$ if $V(\mathcal{F})\subseteq V(\mathcal{H})$ and $E(\mathcal{F})\subseteq E(\mathcal{H})$ and is denoted by $\mathcal{F}\subseteq \mathcal{H}$. The \emph{degree} of a vertex $v\in \mathcal{H}$ with respect to subhypergraph $\mathcal{F}$ is defined as $\deg_{\mathcal{F}}(v):=|\{e\in E(\mathcal{F})\colon v\in e\}|$. The \emph{minimum degree} and \emph{maximum degree} of a hypergraph $\mathcal{H}$ are denoted by $\delta(\mathcal{H})$ and $\Delta(\mathcal{H})$ respectively. A $t$-matching is a set of $t$ pairwise disjoint edges. If $U$ is a nonempty subset of the vertex set $V(\mathcal{H})$ of a hypergraph $\mathcal{H}$, then the subhypergraph $\langle U\rangle$ of $\mathcal{H}$ \emph{induced} by $U$ is the hypergraph having vertex set $U$ and whose edge set consists of those edges of $\mathcal{H}$ incident to two elements of $U$.

Given two hypergraphs $\mathcal{H}$ and $\mathcal{F}$, a \emph{hypergraph morphism} between $\mathcal{H}$ and $\mathcal{F}$ is a function $f\colon V(\mathcal{H})\rightarrow V(\mathcal{F})$ that preserves edges. An \emph{automorphism} is a morphism between a hypergraph and itself. A hypergraph  $\mathcal{H}$ is called \emph{vertex-transitive} if for any pair of vertices $u,v\in V(\mathcal{H})$ there exists an automorphism $f$ such that $f(v)=u$. 

Let $\mathfrak{F}$ be a family of hypergraphs. A hypergraph $\mathcal{H}$ is $\mathfrak{F}$-free if $\mathcal{H}$ contains no element of $\mathfrak{F}$ as a subhypergraph. We also say that $\mathfrak{F}$ is a \emph{forbidden family of hypergraphs for $\mathcal{H}$}. The \emph{extremal number} for the pair $(n,\mathfrak{F})$ is denoted by
\[ex_k(n, \mathfrak{F}) := max\{e(\mathcal{H}) \colon \mathcal{H}\subseteq \binom{[n]}{k}, \text{ and $\mathcal{H}$ is $\mathfrak{F}$-free}\}.
\]An \emph{extremal hypergraph for the pair $(n,\mathfrak{F})$} is a hypergraph $\mathcal{H}$ with $n$ vertices, $\mathfrak{F}$-free and exactly $ex_k(n,\mathfrak{F})$ edges. If $k=2$ then the definition is equivalent to the usual $ex(n,F)$ defined in extremal graph theory (see \cite{diestel2005graph}). \label{prop} It is easy to check that if $ex_k(n,\mathfrak{F}) $ is a monotone function with respect to $\mathfrak{F}$; i.e., if $\mathfrak{F'}$ is a subfamily of $\mathfrak{F}$, then $ex_k(n,\mathfrak{F})\leq ex_k(n,\mathfrak{F'})$.

The \emph{bounded degree family}, denoted by $\mathfrak{D}_k(p,q)$, is defined as the family of all $k$-uniform hypergraphs with $p$ edges, no isolated vertices and maximum vertex degree $q-1$. That is,
\[ 
\displaystyle \mathfrak{D}_k(p,q):=\left\{ \mathcal{H}\subseteq \binom{[s]}{k}\colon s\in \mathbb{N},e(\mathcal{H})=p, \delta(\mathcal{H})\geq 1, \Delta(\mathcal{H})\leq q-1\right\}.
\]
The handshaking lemma for hypergraphs gives us the inequalities
\[\frac{kp}{q-1}\leq v(\mathcal{H})\leq kp
\]
for each $\mathcal{H}\in \mathfrak{D}_k(p,q)$.

A hypergraph $\mathcal{H}\subseteq \binom{[n]}{k}$ satisfies the $(p,q)$-property if and only if $\mathcal{H}$ is $\mathfrak{D}_k(p,q)$-free. So the $(p,q)$-extremal problem is equivalent to finding $$ex_k(n,\mathfrak{D}_k(p,q)).$$

\begin{definition}Given $n,k,t$, positive integers with $n\geq t  \geq 2$, we define the $k$-uniform hypergraph $\mathcal{F}_k(n,t)\subseteq \binom{[n]}{k}$:
\[
\mathcal{F}_k(n,t):= \left\{A\in \binom{[n]}{k}\colon A\cap [t]\neq \emptyset\right\}.
\] 
Furthermore,
\[\mathfrak{F}_k(n,t,r)=\{\mathcal{H}\subseteq \binom{[n]}{k}\colon \mathcal{F}_k(n,t)\subseteq \mathcal{H}, e(\mathcal{H})=e(\mathcal{F}_k(n,t))+r   \}\]
is the family of  $k$-hypergraphs that contains $\mathcal{F}_k(n,t)$ as a subhypergraph and has exactly $r$ edges with no vertices in $[t]$. The graph $\mathcal{F}_2(n,t)$ is also known as a \emph{split graph} (see \cite{golumbic2004algorithmic}).
\end{definition}

\begin{lemma}\label{ida}
If $\mathcal{H}\in \mathfrak{F}_k(n,t,r)$, then $\mathcal{H}$ satisfies the $(p,q)$-property with $t=\lfloor\frac{p-1}{q-1}\rfloor$, and $r$ is the residue of $\frac{p-1}{q-1}$.
\end{lemma}

\begin{proof}
By hypothesis we have that $p-1=t(q-1)+r$, where $t\in \mathbb{N}$ and $0\leq r<q-1$. All the edges of $\mathcal{H}$ have a vertex in $[t]$ except, say, the edges $\mathcal{E}=\{e_1,e_2,\ldots,e_r\}$. If we take any $p$ edges in $\mathcal{H}$ then at most $r$ edges are in  $\mathcal{E}$. Then the other $p-r=t(q-1)+1$ edges have a vertex in $[t]$, but by the pigeonhole principle one vertex has degree at least $q$.
\end{proof}

 \section{The (p,q)-extremal problem for graphs}
The aim of this section is to prove our main theorem. To simplify the notation we introduce the following lemma.

\begin{lemma}\label{konig}
Let $G$ be a bipartite graph with partition sets $A,B$. If $|A|<|B|$ and $e(G)>(t-1)|B|$, then $G$ has a $t$-matching.
\end{lemma}

\begin{proof}
We proceed by contradiction. Suppose that the maximum matching has at most $t-1$ edges. Then by Konig's theorem (see \cite{diestel2005graph}), the minimum vertex cover has at least $t-1$ vertices; then $e(G)\leq (t-1)|B|$, which is a contradiction.
\end{proof}

Observe that for every $\mathcal{H}\in \mathfrak{F}_k(n,t,r)$, $e(\mathcal{H})=\binom{n}{k}-\binom{n-t}{k}+r$. Then following the notation of Lemma \ref{ida}, we may define the function $\varphi_k(n,p,q):=e(\mathcal{H})$.

Theorem \ref{main} can be restated using the previous definitions as follows:

\begin{thmrestate}{main}
Let $n,p,q$ be natural numbers. For $n\geq 2p^2$ and $p\geq q \geq 3$,  we have that 
\[ex(n,\mathfrak{D}_2(p,q))=\varphi_2(n,p,q)
\]and the extremal graph is any $G\in \mathfrak{F}_2(n,t,r)$  up to isomorphism, where $\displaystyle t=\lfloor \frac{p-1}{q-1} \rfloor$ and $r$ is the residue of $\frac{p-1}{q-1}$.
\end{thmrestate}

\begin{proof}
Let $G$ be a graph with $n\geq 2p^2$ vertices that satisfies the $(p,q)$-property for $p\geq q\geq 3$. By Lemma \ref{ida} we know that  $$\varphi_2(n,p,q)\leq ex(n,\mathfrak{D}_2(p,q)),$$  so we may assume that $G$ is a graph with $e(G)\geq \varphi_2(n,p,q)$ edges.

It is not difficult to show that $\varphi_2(n,p,q)$ is an strictly increasing function with respect to the variable $p$.

The proof is by induction on $p$. If $p=q$ the result is trivial. Suppose our theorem is true for $p-1\geq q$; then we shall prove it for $p$. Since $e(G)\geq \varphi_2(n,p,q)> \varphi_2(n,p-1,q)$, we may conclude that $G$ contains a subgraph $F\in \mathfrak{D}_2(p-1,q)$, not necessarily induced. 

%
 
In order to prove that $e(G)\leq \varphi _{2}(n,p,q)$, it will be enough to show that there is a set $X\subseteq V(G)$ of cardinality $t$ contained in $V(F)$ such that $\langle V(G)- X\rangle$ contains exactly $r$ edges.

We will proceed by showing the following three statements:

\begin{enumerate}[a)]
\item $V(G)- V(F)$ is an independent set of vertices in $G$,	
\item $e(V(F)- X,V(G)- V(F)) =0$, and
\item $\langle V(F)- X\rangle $ contains exactly $r$ edges in $G$.
\end{enumerate}

\noindent \textbf{Claim a)} 

Suppose that there exists an edge $e$ contained in $\langle V(G)- V(F)\rangle$. Then the edges of $F$ plus $e$ contradict the $(p,q)$-property.

\medskip\noindent \textbf{Claim b)}
 
We will start by showing that $e( V(F),V(G)- V(F) )>(t-1)(v(G)- v(F))$. Suppose that $e( V(F),V(G)- V(F) )\leq (t-1)( v(G)- v(F))$. Then 
\[
e(G)=e( F )+e(V(F),V(G)- V(F) )+e(V(G)- V(F) ),
\]
and we have that
\[nt -\binom{t+1}{2}+r=\varphi _{2}(n,p,q)\leq \binom{2(p-1)}{2}+\left (t-1\right ) \left (n-\frac{2(p-1)}{q-1} \right ),
\]which implies that $ n<2p^2$, but this is a contradiction.

Since $e( V(F),V(G)- V(F))>(t-1)(v(G)- v(F)) $ and $n\geq 2p^2\geq 4(p-1)\geq 2 v(F)$, then by Lemma \ref{ida}, there is a $ t$-matching contained in $E[V(F),V(G)- V(F)] $. Let $\{e_{1},\dots,e_{t}\}$ be such a matching, let $x_{i}$ be the vertex of the edge $e_{i}$ that is in $V(F)$, and let $X=\{x_{1},x_2,\ldots ,x_{t}\}\subseteq V(F)$.

Consider now $Y=\{y\in V(F)\mid y \text{ is adjacent to some vertex of }V(G)- V(F)\}$. Clearly $X\subseteq Y$. Note that $Y$ is an independent set of vertices in $F$, and suppose there is a edge $f=\{y_{1},y_{2}\}\in F,$ where $y_{1}$ and $y_{2}$ belong to $Y$. Then there are edges $f_{1}=\{y_{1},v_{1}\}$ and $f_{2}=\{y_{2},v_{2}\}$, where $v_{1}$ and $v_{2}$ are not in $V(F)$. By replacing the edge $f$ by the edges $f_{1}$ and $f_{2}$ in $F$, we obtain a contradiction to the fact that $G$ satisfies the $(p,q)$-property (in this step it is important that $q\geq 3$). Next we shall prove that $X=Y$; note that for every $y\in Y$, $\deg_{F}(y)=q-1,$ otherwise, since $y$ is adjacent to some vertex of $V(G)- V(F)$, there is an edge $f_{1}=\{y,v\}$, where $v$ is not in $V(F)$, but hence the edges of $F$ plus $f$ contradict the $(p,q)$-property for $G$. If $|Y|\geq t+1$ then 
\[p-1=e(F) \geq (q-1)|Y| \geq (q-1)(t+1)>t(q-1)+r,
\]which is also a contradiction. Then $|Y|=t$, and so $X=Y$.

\medskip\noindent \textbf{Claim c)} 

First note that $\langle V(F)- X\rangle$ contains exactly $r$ edges in $F$. This is because $F$ has $p-1=t(q-1)+r$ edges, $X$ is a set of $t$ independent points in $F$ and for every $x_i \in X$, $\deg_{F}(x_i)=q-1$. Suppose c) is not true. Then there is an edge $f=\{y_{1},y_{2}\}$ in $E(G)- E(F)$ such that $y_{1}$ and $y_{2}$ belong to $V(F)- X$. We have three cases.

\begin{enumerate}[i)]
\item $\deg_{F}(y_{1})<q-1$ and $\deg_{F}(y_{2})<q-1$;

\item $\deg_{F}(y_{1})<q-1$ and $\deg_{F}(y_{2})=q-1$:

\item $\deg_{F}(y_{1})= \deg_{F}(y_{2})=q-1$.
\end{enumerate}

If case i) occurs, the edges of $F$ plus $f$ contradict the $(p,q)$-property. For case ii), note that there is an edge $f_{1}=\{y_{2},x_{i}\}$ in $F,$ because $\langle V(F)- X\rangle$ contains exactly $r$ edges in $F$ and $r<q-1.$ By replacing $f_{1}$ with $f$ and $e_{i}$ in $F,$ we obtain a contradiction to the $(p,q)$-property. Finally, the fact that $\langle V(F)- X\rangle$ contains exactly $r$ edges in $F$ and $r<q-1$ implies in case iii) that there are edges $f_{1}=\{y_{1},x_{i}\}$ and $f_{2}=\{y_{2},x_{j}\}$, both in $F$, where $i\neq j.$ Replace $f_{1}$ and $f_{2}$ by $f,e_{i},e_{j}$ to obtain a contradiction to the $(p,q)$-property. This confirms that $\langle V(F)-X\rangle$ contains exactly $r$ edges in $G$ and consequently that $e(G)\leq \varphi_2(n,p,q)$. We also proved that $G\in \mathfrak{F}_2(n,t,r)$. \end{proof}

\begin{remark*} The bound $2p^2$ can be improved to $\binom{2(p-1)}{2} +\binom{t+1}{2} -\frac{2(p-1)(t-1)}{q-1}-r$.
\end{remark*}

\begin{conjecture}\label{conj}
For $n$ sufficiently large and $p\geq q \geq 2$,
\[ex_k(n,\mathfrak{D}_k(p,q))=\varphi_k(n,p,q).\]
\end{conjecture}

If $k=2$, $p\geq q=2$, the conjecture is true due to a theorem of Erd\H{o}s and Gallai in \cite{erdHos1959maximal}; if $k\geq 2$ the theorem of Frankl in \cite{frankl2013improved} also confirms the conjecture; if $k=3$ the theorems in \cite{furedi2013hypergraph} and \cite{furedi2011exact} show the $(p,3)$-extremal problem; and finally if $p=q$ and $k\geq 2$, the lemma of Frankl in \cite{frankl1976sperner} proves the $(q,q)$-problem.

In order to analyze the $(p,3)$-property for graphs and $n=p$ we need the following result.
\begin{lemma}\label{p3}
If $F$ is a graph such that $v(F)=e(F)\geq 3$, then $F$ has a vertex of degree greater than 3 or all the vertices have degree 2.
\end{lemma}
\begin{proof}
By induction on the number of vertices. For $n=3$ the result is clear.
Suppose that for $n\geq 3$ we have that the statement is true. Let $F$ be a graph with $n+1=v(F)=e(F)$. If $\deg(x)=2$ for all $x\in V(F)$ or there exists a vertex $y$ such that $\deg(y)\geq 3$, then the statement is true, so suppose that we have a vertex $x$ such that $\deg(x)=1$. Then the graph $F-x$ satisfies the induction hypothesis, and then all the vertices of $F-x$ have degree 2 or there exists a vertex of degree greater than 3; in either case we conclude that $F$ has a vertex of degree greater than 3.
\end{proof}
\begin{theorem}For $p\geq 3$,
\[ex(p,\mathfrak{D}_2(p,3))=\binom{p-1}{2}+1.\]
\end{theorem}
\begin{proof}
We prove the first inequality $ex(p,\mathfrak{D}_2(p,3))\geq\binom{p-1}{2}+1$ by showing that the graph $G= K_{p-1}+e$ (the complete graph $K_{p-1}$ with an extra edge) satisfies the $(p,3)$-property. Any set of $p$ edges in $G$ generates a subgraph $F$ with $p$ vertices and $p$ edges which either has a vertex of degree greater than 3 or all the vertices have degree 2 by Lemma \ref{p3}, but the latter is not possible because we have a vertex of degree 1 in $G$, so $F$ has a vertex of degree 3.
Then $G$ satisfies the $(p,q)$-property and $ex(p,\mathfrak{D}_2(p,3))\geq \binom{p-1}{2}+1$.

For the other inequality we observe that the forbidden graph family for the $(p,3)$-extremal problem is $\displaystyle \mathfrak{D}_2(p,3)$, so the elements of this family are cycles, paths or combinations of both, always with exactly $p$ edges. Then it is clear that the cycle $C_p$ with $p$ edges is a element of $\mathfrak{D}_2(p,3)$ and by the monotone property of $ex(n,\mathfrak{F})$ and Ore's theorem (see \cite{ore1961arc}) we prove that $ex(p,\mathfrak{D}_2(p,3))\leq ex(p,C_p)=\binom{p-1}{2}+1$.
\end{proof}

\section{An application of the (p,q)-extremal problem on fractional coloring of hypergraphs}

In this section we will apply the $(p,q)$-extremal problem to find the fractional chromatic number of Kneser hypergraphs. A set system is \emph{$q$-wise disjoint} if any choice of $q$ sets has an empty intersection. If $q=2$ then we are saying that the system is pairwise disjoint. A \emph{k-coloring} of the vertex set of a hypergraph is a partition of the vertices into $k$ classes such that each class has no edges.

\begin{definition}The \emph{q-wise Kneser $p$-uniform hypergraph}, denoted by $K^p_q\binom{[n]}{k}$, is the $p$-uniform hypergraph with the vertex set $V(K_q^p\binom{[n]}{k})=\binom{[n]}{k}$, and $p$ sets form an edge if they are a $q$-wise disjoint family.
\end{definition}

In 1990 K.\ S.\ Sarkaria proved that the chromatic number of the Kneser $q$-wise $p$-hypergraph is given by the following formula (see \cite{sarkaria1990generalized}):
\[\displaystyle \chi \left(K_q^p \binom{[n]}{k}\right) =\left \lceil \frac{n(q-1)-p(k-1)}{p-1}\right\rceil 
\]for $p,q,n,k\in \mathbb{N}$, $p\geq q\geq 2$ and $n\geq k$.

\subsection{Fractional chromatic number for hypergraphs and its equivalence with extremal problems}

The fractional chromatic number is a generalization of the classical chromatic number where each vertex is colored with a set of colors instead of a single color. There are several equivalent definitions of fractional coloring. In this paper we use the one that defines the fractional chromatic number in terms of independent sets. An \emph{independent set} of $V(\mathcal{H})$ is a set with no edges. The \emph{independence number} of an hypergraph $\mathcal{H}$, denoted by $\alpha(\mathcal{H})$, is defined as the maximum cardinality of a independent set; the family of all independent sets is denoted by $I(\mathcal{H})$, and $I(\mathcal{H},v)$ is the family of all independent sets containing the vertex $v$.

Given a hypergraph $\mathcal{H}=(V,E)$, a \emph{fractional coloring} of a hypergraph is a function $f\colon I(\mathcal{H})\rightarrow \mathbb{R^+}$ such that for all $v \in V(\mathcal{H})$,
\[\sum_{S\in I(\mathcal{H},v)} f(S)\geq 1.
\]The \emph{weight} of a fractional coloring is the sum of all its values, and is defined by the formula $p(f)=\sum_{S\in I(\mathcal{H})} f(S)$. The \emph{fractional chromatic number} of a hypergraph is the minimum possible weight for a fractional coloring and is denoted by $\chi_f(\mathcal{H})$. One important observation is that $\chi_f(\mathcal{H})\leq \chi(\mathcal{H})$ for every hypergraph $\mathcal{H}$. For further details on the definition and properties of the fractional coloring of hypergraphs see \cite{godsil2001algebraic} and \cite{scheinerman2011fractional}. One of the classical results of fractional coloring is that if $\mathcal{H}$ is vertex-transitive, then 
	\[	\chi_f(\mathcal{H})= \frac{v(\mathcal{H})}{\alpha(\mathcal{H})}.\]
	It is easy to check that the Kneser hypergraphs $K^p_q\binom{[n]}{k}$ are vertex-transitive.

Thus, using all our previous results we know that
\[\displaystyle \chi_f\left(K^p_q\binom{[n]}{k}\right)=\frac{\binom{n}{k}}{\alpha(K^p_q\binom{[n]}{k})}.
\]
In order to calculate the fractional chromatic number of a Kneser hypergraph, it is thus sufficient to find its independence number. Note that all independent sets $\mathcal{F}$ of $K_q^p\binom{[n]}{k}$ satisfy that for every $p$ elements of $\mathcal{F}$, $q$ of them intersect: this is precisely the $(p,q)$-property for $k$-uniform hypergraphs. Thus determining the independence number of $K^p_q\binom{[n]}{k}$ is equivalent to finding the $(p,q)$-extremal problem. The following is a corollary of Theorem \ref{main}.

\begin{corollary}
Let $n,p,q,t$ be positive integers, where $p-1=(q-1)t+r$, $p\geq q\geq 3$, $0\leq r < q-1$, $n\geq 2p^2$. Then
\[\displaystyle \chi_f\left(K^p_q\binom{[n]}{2}\right)=\frac{\binom{n}{2}}{\binom{n}{2}-\binom{n-t}{2}+r}.
\]\end{corollary}


\subsection*{Acknowledgements}
The authors wish to acknowledge support from CONACyT under projects 166306, 178395 and support from PAPIIT-UNAM 
under projects IN101912, IN112614.  We would like to thank CINNMA for the support given throughout the years.


\begin{thebibliography}{10}

\bibitem{alon1992piercing}
N.~Alon and D.~J. Kleitman.
\newblock Piercing convex sets and the Hadwiger--Debrunner (p,q)-problem.
\newblock \emph{Advances in Mathematics}, 96(1):103--112, 1992.

\bibitem{diestel2005graph}
R.~Diestel.
\newblock Graph Theory. 2005.
\newblock \emph{Grad.\ Texts in Math}, 2005.

\bibitem{eckhoff2003survey}
J.~Eckhoff.
\newblock A survey of the Hadwiger-Debrunner $(p,q)$-problem.
\newblock In \emph{Discrete and Computational Geometry},  (25)347--377.
  Springer, 2003.

\bibitem{erdos1961intersection}
P.~Erd\H{o}s, C.~Ko, and R.~Rado.
\newblock Intersection theorems for systems of finite sets.
\newblock \emph{The Quarterly Journal of Mathematics}, 12(1):313--320, 1961.

\bibitem{erdHos1965problem}
P.~Erd{\H{o}}s.
\newblock A problem on independent r-tuples.
\newblock \emph{Ann.\ Univ.\ Sci.\ Budapest.}, 8:93--95, 1965.

\bibitem{erdHos1959maximal}
P.~Erd{\H{o}}s and T.~Gallai.
\newblock On maximal paths and circuits of graphs.
\newblock \emph{Acta Mathematica Hungarica}, 10(3):337--356, 1959.

\bibitem{frankl1976sperner}
P.~Frankl.
\newblock On Sperner families satisfying an additional condition.
\newblock \emph{Journal of Combinatorial Theory, Series A}, 20(1):1--11, 1976.

\bibitem{frankl2013improved}
P.~Frankl.
\newblock Improved bounds for Erd{\H{o}}s' matching conjecture.
\newblock \emph{Journal of Combinatorial Theory, Series A}, 120(5):1068--1072,
  2013.

\bibitem{frankl2012maximum}
P.~Frankl, V.~R{\"o}dl, and A.~Ruci{\'n}ski.
\newblock On the maximum number of edges in a triple system not containing a
  disjoint family of a given size.
\newblock \emph{Combinatorics, Probability and Computing}, 21(1-2):141--148,
  2012.
  
\bibitem{frankl2012maximumhypergraph}
P. ~Frankl.
\newblock On the maximum number of edges in a hypergraph with given matching number.
\newblock \emph{arXiv:1205.6847}.

\bibitem{furedi2013hypergraph}
Z.~Füredi and T.~Jiang.
\newblock Hypergraph Turán numbers of linear cycles.
\newblock \emph{Journal of Combinatorial Theory, Series A}, 123(1):252--270 2014.

\bibitem{furedi2011exact}
Z.~Füredi, T.~Jiang, and R.~Seiver.
\newblock Exact solution of the hypergraph Turán problem for $ k$-uniform linear paths.
\newblock \emph{Combinatorica}, 1--24, 2014.

\bibitem{godsil2001algebraic}
C.~D. Godsil and G.~Royle.
\newblock \emph{Algebraic Graph Theory}.
\newblock Springer New York, 2001.

\bibitem{golumbic2004algorithmic}
M.~C. Golumbic.
\newblock \emph{Algorithmic Graph Theory and Perfect Graphs}, volume~57.
\newblock Elsevier, 2004.

\bibitem{hadwiger1957variante}
H.~Hadwiger and H.~Debrunner.
\newblock {\"U}ber eine Variante zum Hellyschen Satz.
\newblock \emph{Archiv der Mathematik}, 8(4):309--313, 1957.

\bibitem{huang2012size}
H.~Huang, P.-S. Loh, and B.~Sudakov.
\newblock The size of a hypergraph and its matching number.
\newblock \emph{Combinatorics, Probability and Computing}, 21(03):442--450,
  2012.

\bibitem{kneser1955aufgabe}
M.~Kneser.
\newblock Aufgabe 360.
\newblock \emph{Jahresbericht der Deutschen Mathematiker-Vereinigung}, (58)2:27,
  1955.

\bibitem{lovasz1978kneser}
L.~Lov{\'a}sz.
\newblock Kneser's conjecture, chromatic number, and homotopy.
\newblock \emph{Journal of Combinatorial Theory, Series A}, 25(3):319--324,
  1978.

\bibitem{luczak2012erdos}
T.~Luczak and K.~Mieczkowska.
\newblock On Erd{\H{o}}s' extremal problem on matchings in hypergraphs.
\newblock \emph{Journal of Combinatorial Theory, Series A}, 124:178--194, 2012.

\bibitem{ore1961arc}
O.~Ore.
\newblock Arc coverings of graphs.
\newblock \emph{Annali di Matematica Pura ed Applicata}, 55(1):315--321, 1961.

\bibitem{sarkaria1990generalized}
K.~S. Sarkaria.
\newblock A generalized Kneser conjecture.
\newblock \emph{Journal of Combinatorial Theory, Series B}, 49(2):236--240,
  1990.

\bibitem{scheinerman2011fractional}
E.~R. Scheinerman and D.~H. Ullman.
\newblock \emph{Fractional Graph Theory}.
\newblock John Wiley and Sons, 1997.

\end{thebibliography}
\end{document}